\documentclass[12pt]{amsart}
\usepackage{amsmath,amssymb,amsbsy,amsfonts,latexsym,amsopn,amstext,cite,
  amsxtra,euscript,amscd,bm}
\usepackage{url}

\usepackage{mathrsfs}
\usepackage{enumitem}
\usepackage{color}
\usepackage[colorlinks,linkcolor=blue,anchorcolor=blue,citecolor=blue,backref=page]{hyperref}
\usepackage{color}
\usepackage{graphics,epsfig}
\usepackage{graphicx}
\usepackage{float}
\usepackage{epstopdf}
\hypersetup{breaklinks=true}

\usepackage[np]{numprint}
\npdecimalsign{\ensuremath{.}}

\definecolor{olive}{rgb}{0.3, 0.4, .1}

\usepackage{tikz}
\usepackage{pgfplots}
\usetikzlibrary{patterns}

\def\Th1{\varTheta}

\usepackage{bibentry}

\usepackage[english]{babel}
\usepackage{mathtools}
\usepackage{todonotes}
\usepackage{empheq}
\usepackage{url}
\usepackage[colorlinks,linkcolor=blue,anchorcolor=blue,citecolor=blue,backref=page]{hyperref}

\usepackage[norefs,nocites]{refcheck}

\DeclareMathOperator\supp{supp}

\begin{document}

\newtheorem{theorem}{Theorem}
\newtheorem{lemma}[theorem]{Lemma}
\newtheorem{claim}[theorem]{Claim}
\newtheorem{cor}[theorem]{Corollary}
\newtheorem{conj}[theorem]{Conjecture}
\newtheorem{prop}[theorem]{Proposition}
\newtheorem{definition}[theorem]{Definition}
\newtheorem{question}[theorem]{Question}
\newtheorem{example}[theorem]{Example}
\newcommand{\hh}{{{\mathrm h}}}
\newtheorem{remark}[theorem]{Remark}

\numberwithin{equation}{section}
\numberwithin{theorem}{section}
\numberwithin{table}{section}
\numberwithin{figure}{section}

\def\sssum{\mathop{\sum\!\sum\!\sum}}
\def\ssum{\mathop{\sum\ldots \sum}}
\def\iint{\mathop{\int\ldots \int}}

\newcommand{\diam}{\operatorname{diam}}

\def\squareforqed{\hbox{\rlap{$\sqcap$}$\sqcup$}}
\def\qed{\ifmmode\squareforqed\else{\unskip\nobreak\hfil
    \penalty50\hskip1em \nobreak\hfil\squareforqed
    \parfillskip=0pt\finalhyphendemerits=0\endgraf}\fi}%%

% use the AMS-Euler Fraktur fonts
%%%%%%%%%%%%%%%%%%%%%%%%%%%%%%%%%% 
\newfont{\teneufm}{eufm10}
\newfont{\seveneufm}{eufm7}
\newfont{\fiveeufm}{eufm5}
%%%%%%%%%%%%%%%%%%%%%%%%%%%%%%%%% 
% 
% allow automatic size selection in math mode
% 
%%%%%%%%%%%%%%%%%%%%%%%%%%%%%%%%% 
\newfam\eufmfam
\textfont\eufmfam=\teneufm
\scriptfont\eufmfam=\seveneufm
\scriptscriptfont\eufmfam=\fiveeufm
%%%%%%%%%%%%%%%%%%%%%%%%%%%%%%%%% 
% 
% \frak works on a single symbol at a time...
% 
\def\frak#1{{\fam\eufmfam\relax#1}}

\newcommand{\bflambda}{{\boldsymbol{\lambda}}}
\newcommand{\bfmu}{{\boldsymbol{\mu}}}
\newcommand{\bfxi}{{\boldsymbol{\eta}}}
\newcommand{\bfrho}{{\boldsymbol{\rho}}}

\def\eps{\varepsilon}

\def\fK{\mathfrak K}
\def\fT{\mathfrak{T}}
\def\fL{\mathfrak L}
\def\fR{\mathfrak R}

\def\fA{{\mathfrak A}}
\def\fB{{\mathfrak B}}
\def\fC{{\mathfrak C}}
\def\fM{{\mathfrak M}}
\def\fS{{\mathfrak  S}}
\def\fU{{\mathfrak U}}
\def\fW{{\mathfrak W}}

\def\T {\mathsf {T}}
\def\Tor{\mathsf{T}_d}
\def\Tore{\widetilde{\mathrm{T}}_{d} }

\def\sM {\mathsf {M}}

\def\ss{\mathsf {s}}

\def\Kmnd{\cK_d(m,n)}
\def\Kmnp{\cK_p(m,n)}
\def\Kmnq{\cK_q(m,n)}

\def \balpha{\bm{\alpha}}
\def \bbeta{\bm{\beta}}
\def \bgamma{\bm{\gamma}}
\def \bdelta{\bm{\delta}}
\def \bzeta{\bm{\zeta}}
\def \blambda{\bm{\lambda}}
\def \bchi{\bm{\chi}}
\def \bphi{\bm{\varphi}}
\def \bpsi{\bm{\psi}}
\def \bnu{\bm{\nu}}
\def \bomega{\bm{\omega}}

\def \bell{\bm{\ell}}

\def\eqref#1{(\ref{#1})}

\def\vec#1{\mathbf{#1}}

\newcommand{\abs}[1]{\left| #1 \right|}

\def\Zq{\mathbb{Z}_q}
\def\Zqx{\mathbb{Z}_q^*}
\def\Zd{\mathbb{Z}_d}
\def\Zdx{\mathbb{Z}_d^*}
\def\Zf{\mathbb{Z}_f}
\def\Zfx{\mathbb{Z}_f^*}
\def\Zp{\mathbb{Z}_p}
\def\Zpx{\mathbb{Z}_p^*}
\def\cM{\mathcal M}
\def\cE{\mathcal E}
\def\cH{\mathcal H}

\def\le{\leqslant}

\def\ge{\geqslant}

\def\sfB{\mathsf {B}}
\def\sfC{\mathsf {C}}
\def\sfS{\mathsf {S}}
\def\L{\mathsf {L}}
\def\FF{\mathsf {F}}

\def\sE {\mathscr{E}}
\def\sS {\mathscr{S}}
\def\sL {\mathscr{L}}

%%%%%%%%%%%%%%%%%%%%%%%%% 
% Alphabet calligraphie %
%%%%%%%%%%%%%%%%%%%%%%%%% 
\def\cA{{\mathcal A}}
\def\cB{{\mathcal B}}
\def\cC{{\mathcal C}}
\def\cD{{\mathcal D}}
\def\cE{{\mathcal E}}
\def\cF{{\mathcal F}}
\def\cG{{\mathcal G}}
\def\cH{{\mathcal H}}
\def\cI{{\mathcal I}}
\def\cJ{{\mathcal J}}
\def\cK{{\mathcal K}}
\def\cL{{\mathcal L}}
\def\cM{{\mathcal M}}
\def\cN{{\mathcal N}}
\def\cO{{\mathcal O}}
\def\cP{{\mathcal P}}
\def\cQ{{\mathcal Q}}
\def\cR{{\mathcal R}}
\def\cS{{\mathcal S}}
\def\cT{{\mathcal T}}
\def\cU{{\mathcal U}}
\def\cV{{\mathcal V}}
\def\cW{{\mathcal W}}
\def\cX{{\mathcal X}}
\def\cY{{\mathcal Y}}
\def\cZ{{\mathcal Z}}
\newcommand{\rmod}[1]{\: \mbox{mod} \: #1}

\def\cg{{\mathcal g}}

\def\1{\mathbf 1}
\def\vy{\mathbf y}
\def\vr{\mathbf r}
\def\vx{\mathbf x}
\def\va{\mathbf a}
\def\vb{\mathbf b}
\def\vc{\mathbf c}
\def\ve{\mathbf e}
\def\vh{\mathbf h}
\def\vk{\mathbf k}
\def\vm{\mathbf m}
\def\vz{\mathbf z}
\def\vu{\mathbf u}
\def\vv{\mathbf v}

\def\vM{\mathbf M}
\def\vN{\mathbf N}

\def\e{{\mathbf{\,e}}}
\def\ep{{\mathbf{\,e}}_p}
\def\eq{{\mathbf{\,e}}_q}

\def\Tr{{\mathrm{Tr}}}
\def\Nm{{\mathrm{Nm}}}

\def\SS{{\mathbf{S}}}

\def\lcm{{\mathrm{lcm}}}

\def\0{{\mathbf{0}}}

\def\({\left(}
  \def\){\right)}
\def\fl#1{\left\lfloor#1\right\rfloor}
\def\rf#1{\left\lceil#1\right\rceil}
\def\fl#1{\left\lfloor#1\right\rfloor}
\def\ni#1{\left\lfloor#1\right\rceil}
\def\sumstar#1{\mathop{\sum\vphantom|^{\!\!*}\,}_{#1}}

\def\mand{\qquad \mbox{and} \qquad}

\def\tblue#1{\begin{color}{blue}{{#1}}\end{color}}

%%%%%%%%%%%%%%%%%%%%%%%%%%%%%%%%%%%%%%%%%%%%%%%%%%%%%%%% 
%%%%%%%%%%%%%%%%%%%%%%%%%%%%%%%%%%%%%%%%%%%%%%%%%%%%%%%% 
%%%%%%%%%%%%%%%%%%%%%%%%%%%%%%%%%%%%%%%%%%%%%%%%%%%%%%%% 
%%%%%%%%%%%%%%%%%%%%%%%%%%%%%%%%%%%%%%%%%%%%%%%%%%%%%%%% 

%%%%%%% END OF STANDARD STUFF %%%%%%%%%

%%%%%%%%%%%%%%%%%%%%%%%%%%%%%%%%%%%%%%%%%%%%%%%%%%%%%%%% 
%%%%%%%%%%%%%%%%%%%%%%%%%%%%%%%%%%%%%%%%%%%%%%%%%%%%%%%% 
%%%%%%%%%%%%%%%%%%%%%%%%%%%%%%%%%%%%%%%%%%%%%%%%%%%%%%%% 
%%%%%%%%%%%%%%%%%%%%%%%%%%%%%%%%%%%%%%%%%%%%%%%%%%%%%%% 
%%%%%%%%%%% 
%%% Spell

\hyphenation{re-pub-lished}

\mathsurround=1pt

\def\bfdefault{b}

\def \F{{\mathbb F}}
\def \K{{\mathbb K}}
\def \N{{\mathbb N}}
\def \Z{{\mathbb Z}}
\def \P{{\mathbb P}}
\def \Q{{\mathbb Q}}
\def \R{{\mathbb R}}
\def \C{{\mathbb C}}
\def\Fp{\F_p}
\def \fp{\Fp^*}

\def \xbar{\overline x}

\title[Exponential Sums over Smooth Numbers]{Exponential sums  over integers without large prime divisors}

\author[S. Drappeau]{Sary Drappeau}
\address{
  Institut de Math{\'e}matiques de Marseille
  Case 907, Campus de Luminy, 
  13288 Marseille Cedex 9, France}
\email{sary-aurelien.drappeau@univ-amu.fr}

\author[I.~E.~Shparlinski]{Igor E. Shparlinski} \address{School of Mathematics and Statistics, University of New South Wales.
  Sydney, NSW 2052, Australia}
\email{igor.shparlinski@unsw.edu.au}

\begin{abstract}
  We obtain a new bound on exponential sums over integers without large prime divisors,  
  improving that of Fouvry and Tenenbaum (1991). 
  For a fixed integer $\nu\ne 0$, we also obtain new bounds on exponential sums with $\nu$-th
  powers of such integers.
  The improvement is based on exploiting more precisely the factorisation of integers
  without large prime divisors, along with existing Type~I and Type~II bounds. For~$\nu=1$
  we use the classical bounds of Vinogradov (1937), while for~$\nu\neq 1$ we use bounds of
  Vaughan (1975) as well as of Fouvry, Kowalski and Michel (2014).
\end{abstract}


\maketitle

\setcounter{tocdepth}{1}
\tableofcontents

\section{Introduction}

Let $P(n)$ denote the largest prime divisor of an integer $n\ge 1$,
with the convention that $P(1) = 1$. 

We recall that an integer $n$ is called {\it $y$-smooth\/} or {\it $y$-friable\/} if  $P(n) \le y$,
see~\cite{Granv, HilTen} for a background.  
% We note that the term  {\it $y$-friable\/} has also been used
% in the same meaning. While this term may be more precise and descriptive, 
% we continue to use the more traditional terminology.

For $x \ge y \geq 2$, we consider the set 
\[
  \cS(x,y) = \{n \in [1, x]\cap \Z:~ P(n) \le y\}
\]
the set of $y$-smooth positive integers  $n \le x$
and as usual we  use $\Psi(x,y) = \# \cS(x,y)$ to denote its cardinality.

For  real numbers $x \ge y \ge 2$ and $\vartheta$ we define the
exponential sum
\[
  T_{\vartheta}(x,y)  =  \sum_{n \in \cS(x,y)} \eq\(\vartheta n\), 
\]
where  $\e(z) = \exp(2\pi i z)$.  
Sums of these kind, and their generalisations with  non-linear functions of $n$,
have a long history of studying, in particular in relation to Waring's problem and the circle method, see for example~\cite{dlB1, dlB2, dlBGr1, dlBTen, Drap,  FoTe, Harp, MaWa} and the references therein. 

The bound of Fouvry and Tenenbaum~\cite[Theorem~13]{FoTe} (after replacing some logarithmic factors with $x^{o(1)}$) asserts that for $y \le x^{1/2}$,  uniformly over  integers $a$ and $q$
with $\gcd(a,q)=1$, we have 
\begin{equation}
  \label{eq: FT Real}
  \left|T_{\vartheta}(x,y)\right|  \le x^{1+o(1)} \(x^{-1/4} y^{1/2} +  q^{-1/2}+ \(x/qy\)^{-1/2}\) \sL , 
\end{equation}
where
\begin{equation}
  \label{eq: L aq}
  \sL = 1 + x\abs{\vartheta - \frac{a}{q}}.
\end{equation}

In turn, the bound~\eqref{eq: FT Real} follows from the bound
\begin{equation}
  \label{eq: FT Rat}
  \left|S_{a,q}(x,y) \right|   \le x^{1+o(1)} \(x^{-1/4} y^{1/2} +  q^{-1/2}+  \(x/qy\)^{-1/2}\) 
\end{equation}
on rational sums 
\[
  S_{a,q}(x,y)  = T_{a/q}(x,y)=  \sum_{n \in \cS(x,y)} \eq\(an\), 
\]
where  $\eq(z) = \exp(2\pi i z/q)$.

It is easy to see  that the bound~\eqref{eq: FT Rat} is trivial unless $x \ge q^{1+\varepsilon} y$
and $y \le x^{1/2-\varepsilon}$ for some fixed $\varepsilon>0$.
% Here, we exploit the trilinear structure of the 
% exponential sums, which arise in the proof of~\cite[Theorem~13]{FoTe} 
% (following some ideas  from~\cite{PetShp}),
Here, we exploit more precisely the bilinear structure of the exponential sums, and obtain a bound which is 
nontrivial starting from  $x \ge q^{1+\eps}$ which is  clearly an optimal range (apart from the presence of $\eps>0$).

In order to simplify the exposition, we concentrate on the regime when
$y$ grows as some power $x$, say $y = x^{\eta +o(1)}$, for some fixed $\eta > 0$, and in particular we have 
$\Psi(x,y) \ge c(\eta) x$, for some constant $c(\eta)>0$, depending only on $\eta$, see~\cite{Granv,HilTen}.
Hence, in this range,  
the trivial bound $T_{\vartheta}(x,y)$, which we try to improve,  is essentially 
$|T_{\vartheta}(x,y)|  \le x$. 
A more careful examination of our argument, with full book-keeping of all logarithmic factors and  invoking better bounds on the divisor function ``on average'', is most likely able to  lead to new 
bounds also in the range when $y = x^{o(1)}$.

\begin{theorem}\label{thm: Sum S}
  Let~$\eps>0$. For all $x \ge y\ge 2$, and all integers $a$ with $\gcd(a,q)=1$, we have
  \[
    \left|S_{a,q}(x,y) \right| %% \ll_\eps x^{1+\eps} 
    \le x^{1+o(1)} \( \min\{x^{-1/5}, (x/y)^{-1/4}\} + q^{-1/2} + (x/q)^{-1/2}\). 
  \]
\end{theorem}

The saving~$x^{-1/5}$ corresponds to the classical Vinogradov bound~$x^{4/5}$ on exponential sums over primes~\cite{Vinog45}.

Using partial summation, as in the proof of~\cite[Theorem~13]{FoTe}, 
we now estimate the sums $T_{\vartheta}(x,y)$.

\begin{cor}\label{cor: Sum T}
  Uniformly for $x \ge y \ge 2$ and all real $\vartheta$, we have
  \[
    \left|T_{\vartheta}(x,y)\right|   
    \le x^{1+o(1)} \( \min\{x^{-1/5}, (x/y)^{-1/4}\}  +  q^{-1/2} + (x/q)^{-1/2}\) \sL,
   \]
  where $\sL$ is given by~\eqref{eq: L aq}. 
\end{cor}

A variety of other bounds can be found in~\cite{dlB1, dlB2, dlBGr1, dlBTen, Drap,  Harp, MaWa}, 
which improve~\eqref{eq: FT Real}  and~\eqref{eq: FT Rat}  (especially for small $y$), however they  do not seem to affect our improvement.

We also recall that for the classical Waring problem and also similar Waring type problems 
exponential sums with powers $n^\nu$ have also been considered, see~\cite{BrWo-1,BrWo-2,DrSh,Vau,Wool1,Wool2} and references therein. 
For rational  exponential sums with other non-linear functions over the integers 
$n \in \cS(x,y)$ see~\cite{Gong,QiZh}.
Motivated by  these results we consider the rational exponential sums
\[
  S_{\nu,a,q}(x,y)  =  \sum_{n \in \cS(x,y)} \eq\(an^\nu\)
\]
with $\nu$-th powers of smooth numbers, where~$\nu\in \Z\setminus \{0\}$. 
Combining our approach to proving Theorem~\ref{thm: Sum S}
with some bounds from~\cite{Mac} we obtain new estimates on these sums too.
We however have to assume that $q$ is prime. 

First we observe that the bound~\eqref{eq: FT Rat} can be extended 
to the sums $S_{\nu,a,q}(x,y)$ (at least for a prime $q$, and to be 
nontrivial this generalisation still requires $x \ge q^{1+\varepsilon} y$ with some 
fixed $\varepsilon > 0$).
We obtain a bound which gives a power saving for smaller values of $x$.

\begin{theorem}\label{thm: Sum Snu} Let $\nu \ne 0$ be a fixed positive integer, and let~$\eps>0$.
  There exists~$\delta>0$, which depends only on $\nu$ and $\eps$ and  such that the following holds.
  Assume that a prime $q\ge 1$ and real $x \ge y\ge 2$ satisfy~$q\le x^2$.
  Then, uniformly over $a$ with $\gcd(a,q)=1$, we have the estimates:
  \[ S_{\nu, a, q}(x, y) \le x^{1+o(1)} \min\left\{E_1, E_2, E_3, E_4\right\}, \]
  where
  \begin{subequations}\begin{align}
      & E_1= (x/y)^{-1/4} + q^{-1/2} + (x/q)^{-1/2}, \label{Snu-bilin}\\
      & E_2=  y^{-1/2} + x^{-1/4}q^{1/8} + q^{-1/2} + (x/q)^{-1/2}, \label{Snu-bilin+typeI} \\
      & E_3= \min\left\{(x/q)^{-1/4}, (x/y)^{-1/4}q^{1/8}\right\} + q^{-1/4} + (x/y)^{-1/4}, \label{Snu-FM}\\
      & E_4= \(q^{-1/4} + q^{3/4+\eps} x^{-1}\)^\delta. \label{Snu-FKM}
    \end{align}
  \end{subequations}
\end{theorem}

The estimates~\eqref{Snu-bilin} and \eqref{Snu-bilin+typeI} are non-trivial only inside to~$q<x$, but are numerically better in most of that range.
The bounds~\eqref{Snu-FM} and \eqref{Snu-FKM} hold true upon replacing the function~$n\mapsto \eq(an^\nu)$ by any non-exceptional trace function, in the terminology of~\cite{FKM}, unlike the estimates~\eqref{Snu-bilin} and \eqref{Snu-bilin+typeI} which use the morphism property of monomials. 
The regions where these bounds are non-trivial are drawn in Figure~\ref{fig:range-nontriv}.

\begin{figure}[h]
  \centering
  \begin{tikzpicture}[scale=4]
    \fill[pink!50] (0,0) -- (1/3, 1/3) -- (1/3, 2/3) -- (1/5, 4/5) -- (0, 2/3) -- cycle;
    \fill[olive!50] (0,0) -- (1,0) -- (1,1) -- (1/2,1) -- (1/5,4/5) -- (1/3,2/3) -- (1/3,1/3) -- cycle;
    \fill[blue!50] (1/2,1) -- (1,1) -- (1,4/3) -- (1/3,4/3) -- cycle;
    \fill[red!50] (0,2/3) -- (1/2,1) -- (0,2) -- cycle;

    \node at (1/6, 1/2) {$E_1$};
    \node at (2/3, 1/2) {$E_2$};
    \node at (1/5, 6/5) {$E_3$};
    \node at (3/4, 1.15) {$E_4$};

    \draw[dashed] (1/3,4/3) -- (0, 4/3) node[left] {$4/3$}; %% {$\simeq 4/3$};

    \draw[->] (-0.1,0) -- (1.2,0) node[right] {$\alpha$};
    \draw[->] (0,-0.1) -- (0,2.2) node[left] {$\beta$};
    \draw (1,0.05) -- (1,-0.05) node[below] {$1$};
    \draw (0.05,1) -- (-0.05,1) node[left] {$1$};
    \draw (0.05,2) -- (-0.05,2) node[left] {$2$};
    \node[below left] at (0, 0) {$0$};
  \end{tikzpicture}
  \caption{Ranges where the bounds from Theorem~\ref{thm: Sum Snu} are relevant. Here~$y=x^\alpha$ and~$q=x^\beta$.}
  \label {fig:range-nontriv}
\end{figure}

We remark that $\nu$ in Theorem~\ref{thm: Sum Snu} below can also be chosen negative, granted we restrict the sum to integers coprime with~$q$, 
that is, to $n$ for which $n^\nu$ is well-defined modulo $q$.

In particular, we get the following bound which summarises the non-trivial range allowed by Theorem~\ref{thm: Sum Snu}: 

\begin{cor}\label{cor: Sum Snu} Let $\nu \ne 0$ be a fixed positive integer, and let~$\eps>0$.
  There exists~$\delta>0$, which depends only on $\nu$ and $\eps$ and  such that
for a prime $q\ge 1$ and real $x \ge y\ge 2$
  \[  S_{\nu, a, q}(x, y) \ll x^{1-\delta}  \]
  holds uniformly in the range
  \[  %% 1\le x \le y \mand 
  x^\eps \le q \le \max\left\{x^{4/3-\eps}, x^{2-\eps}y^{-2}\right\}.  \]
\end{cor}

% \begin{theorem}\label{thm: Sum Snu} Let $\nu \ne 0$ be a fixed positive integer. 
%   Assume that a  prime $q\ge 1$ and  real $x \ge y\ge 2$ satisfy
%   \[
%     x  < \frac{1}{4} q^{4/3}.
%   \] 
%   Then, uniformly over $a$ with $\gcd(a,q)=1$, we have
%   \[
%     \left|S_{\nu,a,q}(x,y) \right|   \le  \begin{cases}  q^{5/48+o(1)}x^{29/32} y^{1/32}&
%       \text{for $y > q^{1/2}$,}\\
%       q^{5/48+o(1)}x^{29/32} y^{3/104}  &
%       \text{for $y \le q^{1/2}$.} 
%     \end{cases} 
%   \] 
% \end{theorem}

These estimates improve the range $x \ge q^{1+\varepsilon} y$, accessible via the approach of~\cite{FoTe}. We also improve the dependency in~$y$ in the estimate
\[
  %% \begin{equation}
  %%   \label{eq: BW-bound}
  \left|S_{\nu,a,q}(x,y) \right|  \le xq^{-1/4+o(1)} + q^{1/8} x^{3/4+o(1)} y^{1/2}
  %% \end{equation}
\]
implied by a result of Br{\"u}dern  and Wooley~\cite[Theorem~1.1]{BrWo-2}; compare with~\eqref{Snu-FM}.
We also note that our argument applies to more general sums twisted by multiplicative 
functions, see Section~\ref{sec:com}.

\section{Bounds on multilinear exponential sums}

\subsection{Preliminaries} 

We recall that  the notations $U = O(V)$, $U \ll V$ and $ V\gg U$  
are equivalent to $|U|\leqslant c V$ for some positive constant $c$, 
which throughout this work is absolute, unless indicated otherwise. Furthermore, we use $U\asymp V$ 
in the case when $U \ll V \ll U$.

We also write $U = V^{o(1)}$ if for any fixed $\varepsilon$ we have 
$V^{-\varepsilon} \le |U |\le V^{\varepsilon}$ provided that~$V$ is large 
enough. 

For an integer $\ell \ne 0$ we denote by $\tau(\ell )$ the number of positive integer  divisors of~$\ell $, 
for which we very often use the well-known bound
\begin{equation}
  \label{eq:tau}
  \tau(\ell ) = |\ell |^{o(1)} 
\end{equation}
as $|\ell | \to \infty$, see~\cite[Equation~(1.81)]{IwKow}.

The letter $p$, with or without subscripts, always denotes a prime number.

Finally, to simplify the statements of our results, we use the notation  
\[
  U \lesssim V
\]
to denote that $|U| \le V x^{o(1)}$ as $x\to \infty$. 

\subsection{Bounds arising from the theory of trace function}

We recall two bounds arising from the theory of trace functions. We recall the definition of a non-exceptional trace function~$K:\Z\to\C$ in~\cite[Definition~1.3]{FKM}. In particular,  for any prime~$q$ and~$\gcd(a, q)=1$ maps of the form~$n\mapsto \eq(a n^\nu)$ are trace functions, which are non-exceptional if~$\nu\not\in\{0, 1\}$. See~\cite[Remark~1.4]{FKM} for this and other concrete examples of 
trace functions. The precise definition of exceptional trace functions is given after~\cite[Remark~1.4]{FKM}, and in particular 
excludes the case $\nu=1$.

The first result is~\cite[Theorem~1.17]{FKM}, and concerns Type~II sums  for trace functions. 

\begin{lemma} \label{lem:FKM-bilin}
  Let~$K$ be a non-exceptional trace function.
  Let~$q$ be a prime, and~$(\alpha_m), (\beta_n)$ be bounded sequences supported on integers coprime with~$q$ in the intervals~$[M, 2M]$ and~$[N, 2N]$ respectively, where~$M, N\ge 1$ and~$MN\asymp x$. Then
  \[  \sum_{M \le m \le 2M}  \sum_{N \le n \le 2N} \alpha_m \beta_n K(mn) \lesssim x \( q^{-1/4} + M^{-1/2} + q^{1/4} N^{-1/2}\).  \]
\end{lemma}

The second result concerns special convolution with primes for trace functions. Note that  
in the case~$y>x/2$ is~\cite[Theorem~1.15, Equation~(1.3)]{FKM} (the~$m$-sum below is reduced to
just one term with~$m=1$).

\begin{lemma}
  \label{lem:trace-mp}
  Let~$q$ be a prime, and assume~$1\le y \le x$. For any~$\eps>0$, there exists~$\delta>0$ such that
  \begin{subequations}
    \begin{align}
      &\sum_{y<p\le x} \sum_{m\le x/p} K(mp) \lesssim x q^{-\delta/4} + x^{1-\delta} q^{(3/4+\eps)\delta}, \label{eq:trace-mp} \\
      &\sum_{y<p_1\le p_2\le x} \sum_{m\le x/p_1p_2} K(mp_1p_2) \lesssim x q^{-\delta/4} + x^{1-\delta} q^{(3/4+\eps)\delta}.     \label{eq:trace-mp1p2}
    \end{align}
  \end{subequations}
\end{lemma}

\begin{proof}
  Consider first~\eqref{eq:trace-mp}.  Let~$\Delta\in (1/x, 1]$ be some  parameter to be fixed later. 
  Using a partition of unity as in~\cite[p.~1717]{FKM}, we have an upper-bound
  \begin{equation}\label{eq:Part Unity}
    \sum_{y<p\le x} \sum_{m\le x/p} K(mp) \lesssim x\Delta +
    \abs{\sum_{p} \sum_{m} V_0(mp) V_1(p)  K(mp)} 
  \end{equation}
  for some functions~$V_0, V_1$ which are smooth on~$\R$, with supports
  \[  \supp V_0 \subseteq [L, 2L], \qquad \supp V_1 \subseteq [P, 2P],  \]
  for some~$L$ and $P$ satisfying $1 \le L \ll x$ and~$y\ll P \ll x$, and moreover~$V_0, V_1$ have derivatives bounded by~\cite[Equation~(1.1)]{FKM}, that is, 
  \begin{equation}\label{eq:derivatives V}
    x^k  V_{h}^{(k)}(x)\ll Q^{k}, \qquad h = 0,1,
  \end{equation}
  for $k =0,1, \ldots$, 
  with~$Q=\Delta^{-1}$ (all the implied constant throughout the proof may depend $k$).

  In particular,  on the right hand side of~\eqref{eq:Part Unity} and in several formulas below, 
  the sums over $p$ and $m$ are both supported over finite sets.

  Let~$ \widehat V_h$ be the Mellin inverse of $ V_h$. For any fixed $k \ge 0$ we have 
  \begin{equation}\label{eq:Bound Vh-hat}
    \widehat V_h(s) \ll \(\frac{Q}{1+|s|}\)^k,  \qquad h = 0,1.
  \end{equation} 

  Hence
  \[
    V_h(y) = \frac1{2\pi i}\int_{\abs{t}\le \Delta^{-2}} \widehat V(it) y^{-it} dt +  O(\Delta),  \qquad h = 0,1.
  \]  
  Thus, using above representation for $V_1$ and the bound
  \[
    \int_{\abs{t} \le \Delta^{-2}}  \abs{\widehat V_h(it)} dt \ll  Q\log Q   \lesssim  \Delta^{-1} 
  \]
  which follows from~\eqref{eq:Bound Vh-hat} taken with $k=1$  (and our assumption $\Delta > 1/x$) we derive 
  \begin{equation}\label{eq:trace-mp-beforeHB}
    \begin{split}
      \abs{\sum_{p} \sum_{m} V_0(mp) V_1(p)  K(mp)}&
      \\ \lesssim x \Delta  + \Delta^{-1} \sup_{\abs{t} \le \Delta^{-2}} & \abs{\sum_{p} \sum_{m} p^{it} V_0(mp) K(mp)}.
    \end{split}   
  \end{equation}

  Next, we fix $t$ with
  \begin{equation}\label{eq: small t}
    \abs{t} \le \Delta^{-2}
  \end{equation}
  and use the Heath-Brown identity for primes~\cite{HB1} followed by a partition of unity, as in~\cite[Section~4.1]{FKM}. It is also convenient to rename the variable~$m$ in~\eqref{eq:trace-mp-beforeHB} as~$n_1$. Thus, after fixing an arbitrary integer  $J\geq 1$ and setting $J^* = J+1$,  we are reduced to bound sums $\Sigma$ of the shape
  \begin{equation}\label{eq:identity-HB}
    \begin{split}
      \Sigma(\vM,\vN)  & =\underset{m_1, \dotsc, m_J}{\sum\dotsb\sum} \alpha_1(m_1) \dotsb \alpha_J(m_J) \\
      & \qquad \quad \times \underset{n_1, \dotsc, n_{J^*}}{\sum\dotsb\sum} V_1(n_1) V_2(n_2) \dotsb V_{J^*}(n_{J^*}) \\
      & \qquad\qquad \qquad   \times (n_2 \dotsc n_{J^*})^{it}  V_0(m_1 \dotsb m_J n_1 \dotsb n_{J^*}),
    \end{split}  
  \end{equation}
  where
  \begin{itemize}
    \item $\vM = \(M_1, \ldots, M_J\)$  and  $\vN = \(N_1, \ldots, N_{J^*}\)$  are tuples of parameters in 
    $[1/2, 2x]^{2J}$ and $[1/2, 2x]^{2J^*}$, respectively, which satisfy
    \begin{align*}
      N_2 \ge  \ldots \ge  N_{J^*}, \qquad &M_1, \ldots, M_J \le x^{1/J}, \\  
      M_1 \cdots M_J N_1\cdots &N_{J^*} \ll X
    \end{align*}
    with an implied constant depending only of $J$; 
    \item the weights $\alpha_j(m_j)$ are bounded and supported in $[M_j/2, M_j]$, $j =1, \ldots, J$; 
    \item  the maps~$V_j$ are compactly supported in $[N_j/2, N_j]$ and their derivatives satisfy
    \[
      y^k V_j^{(k)}(y) \ll 1, \qquad j =1, \ldots, J^*,
    \]
    for  every integer $k\geq 0$. 
  \end{itemize}
  
  For each~$j=2, \ldots, J^*$, let~$\widetilde V_j(y) = V_j(y) y^{it}$, 
  which, in view of~\eqref{eq: small t}, satisfies 
  \begin{equation}\label{eq:derivatives V-tilde}
    y^k \widetilde V_j^{(k)}(y) \ll \Delta^{-2k}
  \end{equation} 
  for each fixed integer $k\geq 0$. Hence we can rewrite $\Sigma$ as 
  \begin{align*}
    \Sigma(\vM,\vN)   & = \underset{m_1, \dotsc, m_J}{\sum\dotsb\sum} \alpha_1(m_1) \dotsb \alpha_J(m_J) \\
    &\qquad  \quad \times \underset{n_1, \dotsc, n_{J^*}}{\sum\dotsb\sum} V_1(n_1) \widetilde V_2(n_2) \dotsb \widetilde V_{J^*}(n_{J^*}) \\
    &\qquad  \qquad  \qquad \qquad   \quad \times V_0(m_1 \dotsb m_J n_1 \dotsb n_{J^*}). 
  \end{align*}
  Compared with~\cite[Equation~(4.1)]{FKM}, the only difference is the growth of the derivatives of~$\widetilde  V_2, \dotsc, \widetilde V_{J^*}$.
  From here, the arguments in~\cite[Section~4.2]{FKM}, which essentially use only cancellations with respect to two variables 
  of summation $n_1$ and $n_2$.  Recalling~\eqref{eq:derivatives V} and~\eqref{eq:derivatives V-tilde},  we see that we have 
  $Q_U, Q_V, Q_W \le Q^{O(1)}$ in the condition of~\cite[Theorem~1.16]{FKM},  and thus the bound~\cite[Equation~(4.2)]{FKM}
  becomes
  %% adapted with the only change that the   Type~$I_2$ 
  %% deteriorates to
  \[
    \Sigma(\vM,\vN)   \ll    Q^A x \(1 + \frac{q}{N_1N_2}\)^{1/2} q^{-1/8+\eps} 
  \]
  for some~$A $, depending only on $\eps$. Following these arguments, we obtain
  \[  \sum_{y<p\le x} \sum_{m\le x/p} K(mp) \lesssim x\Delta + q^\eps x \Delta^{-A} (1 + q/x)^{1/6} q^{-1/24},  \]
  from which the desired result follows upon balancing~$\Delta$ (which, as one can easily see, satisfies the condition $\Delta > 1/x$).
  
  The bound~\eqref{eq:trace-mp1p2} follows by an identical argument.
\end{proof}

\subsection{Bounds with monomials: bilinear forms and primes}

The following estimate is Vinogradov's classical result for bilinear exponential sums when~$\nu=1$, extended to~$\nu\neq 0$ through the use of Dirichlet characters.

\begin{lemma}
  \label{lem:sum-bilin}
  Let~$\nu\neq 0$ be an integer. Let $M, N, x \ge 2$, and $(\alpha_m), (\beta_n)$ with~$\abs{\alpha_m}\le 1$ and~$\abs{\beta_n}\le 1$ be sequences supported on integers coprime with~$q$ in dyadic intervals~$[M, 2M]$, $[N, 2N]$ respectively. Then
  \begin{align*}
    \sum_{\substack{mn \le x \\ M\le m \le 2M \\ N \le n \le 2N}}  \alpha_m & \beta_n \e_q(a (m n)^\nu) \\
    & \lesssim MN  \(M^{-1/2} + N^{-1/2} + q^{-1/2} + (MN/q)^{-1/2}\), 
  \end{align*}
  where the coprimality assumption~$\gcd(mn,q)=1$ on the supports of~$(\alpha_m)$ and~$(\beta_n)$ can be removed if~$\nu \ge 1$.
\end{lemma}

\begin{proof} Clearly, we can assume that $MN \le x$ as as otherwise the sum is void.
  Hence all terms of the shape $(MN)^{o(1)}$ can be absorbed in  $ \lesssim$.

  We separate analytically the variables~$m, n$ in the condition~$mn\le x$ by means of~\cite[Lemma~13.11]{IwKow}, getting
  \begin{equation}\label{eq:BilinSum}
    \sum_{\substack{mn \le x \\ M\le m \le 2M \\ N \le n \le 2N}}  \alpha_m \beta_n \e_q(a (m n)^\nu) \lesssim \sup_{t\in\R} \abs{ \sum_{\substack{M\le m \le 2M \\ N \le n \le 2N}}  \alpha_{m,t} \beta_{n,t} \e_q(a (m n)^\nu) },
  \end{equation}
  where~$\alpha_{m,t} = \alpha_m m^{it}$ and similarly for~$\beta_{n,t}$.  We now partition the last sum into  at most 
  $(M/q+1) (N/q+1) \ll q^{-2} \max\{M, q\}\max\{N,q\}$ sums with ranges of variables $m$ and $n$ of length 
  $X = \min\{M, q\}$ and $Y = \min\{N,q\}$, respectively.  By a classical result, see, for example~\cite[Chapter~VI, Exercise~14.a]{Vinog}, each 
  of these sums is bounded by $\sqrt{qXY}$. Since   $\max\{A, B\} \min\{A, B\} = AB$ this leads  to the bound 
  \begin{align*}
    q^{-2} \max\{M, q\}&\max\{N,q\} \sqrt{q \min\{M, q\}  \min\{N.q\}}\\
    & =  q^{-2}  \sqrt{ \max\{M, q\}\max\{N,q\} q^3 M  N}\\
    &\le  (MN)^{1/2} q^{-1/2} (M + q)^{1/2} (N + q)^{1/2}, 
  \end{align*}
  which after substitution in~\eqref{eq:BilinSum} concludes the proof.
\end{proof}

% 
% Then, since by hypothesis the sum runs over~$m, n$ coprime with~$q$, we have
% \[  \e_q(a (m n)^\nu) = \frac1{\varphi(q)} \sum_{\chi \pmod{q}} E(\chi) \chi(m) \chi(n),  \]
% where
% \[  E(\chi) = \sum_{\substack{b\pmod{q} \\ \gcd(b, q) = 1}} \overline{\chi(b)} \e_q(a b^k).  \]
% By orthogonality of characters and the classical bound for Gauss sum (see \emph{e.g.}~\cite[Lemma~3.1]{DrSh}) we get~$E(\chi) = O(q^{1/2})$ uniformly in~$\chi$, and therefore
% \[  \sum_{\substack{M\le m \le 2M \\ N \le n \le 2N}}  \alpha_{m,t} \beta_{n,t} \e_q(a (m n)^\nu) \ll \frac{q^{1/2}}{\varphi(q)} \sum_{\chi\pmod{q}} \abs{A(\chi)} \abs{B(\chi)},  \]
% where~$A(\chi) = \sum_{M\le m \le 2M} \alpha_{m, t} \chi(m)$ and similarly for~$B(\chi)$. By orthogonality of characters, we readily have have
% \[  \sum_{\chi\pmod{q}} \abs{A(\chi)}^2 \ll M(M + q)  \]
% and similarly for~$B$, whence by the Cauchy-Schwartz inequality we deduce
% \[  \sum_{\substack{M\le m \le 2M \\ N \le n \le 2N}}  \alpha_{m,t} \beta_{n,t} \e_q(a (m n)^\nu) \lesssim (MN)^{1/2} q^{-1/2} (M + q)^{1/2} (N + q)^{1/2}.  \]
% Our claimed bound follows upon using~$MN \le x$. .
% \end{proof}

The following result  is a variant of the classical bound for exponential sums over primes, with an additional convolution.

\begin{lemma}
  \label{lem:sum-p}
  For any $2 \le y \le x$, we have
  \[ \sum_{y < p \le x} \sum_{m \le x/p} \eq(a m p) \lesssim x^{4/5} + x^{1/2}(x/q + q)^{1/2}. \]
\end{lemma}

\begin{proof}
  We split the sum over $p$ as
  \[ \sum_{y < p \le x} \sum_{m \le x/p} \eq(a m p) = S_1 + S_2 \]
  where~$S_1$ is subject to~$p\le x^{4/5}$ and~$S_2$ is the complementary sum. To~$S_1$ we apply~\cite[Equation~(13.46)]{IwKow}
  with the choice $M = x^{4/5}$,  getting the admissible bound
  \[ S_1 \lesssim x^{4/5} + x/q + q. \]
  To evaluate~$S_2$, by partial summation, a trivial bound on the contribution of prime powers, and splitting in dyadic intervals, we get
  \[
    S_2 \ll x^{3/5} + \sup_{M\le x^{1/5}}(S_{21}(M) + S_{22}(M)),
  \]
  where
  \begin{align*}
    S_{21}(M) & =  \frac1{\log x} \sum_{M/2<m\le M} \sum_{\substack{y<n \le x/m \\ n>x^{5/6}}} \Lambda(n) \e_q(amn), \\
    S_{22}(M) & = \int_2^x \frac1{t(\log t)^2} \sum_{M/2<m\le M} \sum_{\substack{y<n\le x/m \\ x^{5/6}<n<t}} \Lambda(m) \e_q(amn) dt.
  \end{align*}
  Let us focus on~$S_{21}(M)$. We use the Vaughan identity, see~\cite[Equation~(13.39)]{IwKow} with parameters~$y, z$ there replaced by~$x^{2/5}/M$ and~$x^{2/5}$ respectively, getting
  \[
    S_{21} \log x \ll \abs{\Sigma_1} + \abs{\Sigma_2} + \abs{\Sigma_3}
  \]
  where 
  \begin{align*}
    &\Sigma_1 = \sum_{M/2<m\le M} \sum_{b\le x^{2/5}/M} \sum_{n:~bn\in \cI_m} \mu(b)(\log n) \e_q(ambn),   \\
    &\Sigma_2 = \sum_{M/2<m\le M} \sum_{b\le x^{2/5}/M} \sum_{c\le x^{2/5}} \sum_{n:~bcn\in \cI_m} \mu(b)\Lambda(c) \e_q(ambcn), \\
    &\Sigma_3 =  \sum_{M/2<m\le M} \sum_{b> x^{2/5}/M} \sum_{c > x^{2/5}} \sum_{n:~bcn\in \cI_m} \mu(b)\Lambda(c) \e_q(ambcn), 
  \end{align*}
  and~$\cI_m = \Z \cap (\max\{y, x^{5/6}\}, x/m]$. 
  
  Then we follow the steps in the proof of~\cite[Theorem~13.6]{IwKow}.
  In particular, we note that by~\cite[Equation~(8.6)]{IwKow},  for any $c \in \Z$ and $z > 0$, we have 
  \[
    \sum_{1 \le n \le z} \e_q(cn) \ll  \max\left\{z, \left\|c/q\right\|^{-1}\right\} , 
  \]
  where $\| \xi\| = \min\{|\xi - k|:~k \in \Z\}$. Now, using partial summation, 
  we see that the first sum $\Sigma_1$ is bounded by
  \begin{equation}
    \begin{aligned}
      \Sigma_1 &\ll  \int_1^x \sum_{M/2<m\le M} \sum_{b\le x^{2/5}/M} \left|\sum_{n:~bn \in I_m, n>t} \e_q(ambn)\right| \frac{dt}t \\
      & \ll   \int_1^x \sum_{M/2<m\le M} \sum_{b\le x^{2/5}/M} \max\left\{\frac{x}{bm}, \left\|amb/q\right\|^{-1}\right\} \frac{dt}t \\
      &  \ll  \sum_{\ell \leq x^{2/5}} \tau(\ell) \max\left\{\frac{x}{\ell}, \left\| a\ell/q \right\|^{-1}\right\}.
    \end{aligned}\label{eq:typeI-withconvolution}
  \end{equation}
  Recalling   the bound~\eqref{eq:tau}, we derive 
  \[
    \Sigma_1  \lesssim x^{2/5} + x/q + q.
  \]

  The second sum $ \Sigma_2$  is bounded similarly, with the upper bound on~$\ell$ being~$\ell \leq x^{4/5}$, and thus we have
  \[
    \Sigma_2 \lesssim x^{4/5} + x/q + q.
  \]
  Finally the third sum  $\Sigma_2$  is bounded using Lemma~\ref{lem:sum-bilin} by 
  \[
    \Sigma_3 \lesssim x \((x^{-2/5})^{1/2} + q^{-1/2} + (x/q)^{-1/2}\) \ll  x^{4/5} + \frac{x}{q^{1/2}} + (qx)^{1/2}, 
  \]
  which dominates the above bounds on $\Sigma_1$ and $\Sigma_2$. 
  Combining  these estimates, we deduce
  \[
    S_{21}(M) \lesssim x^{4/5} + \frac{x}{q^{1/2}} + (qx)^{1/2}. 
  \]
  The same upper bound on~$S_{22}(M)$, and therefore on~$S_2$ follows by an identical analysis.
\end{proof}

For convolution with one prime for non-linear phases we get the following slightly worse estimate.

\begin{lemma} %[Convolution with one prime for non-linear phases]
  \label{lemma:sump-monom}
  Let~$\nu\in \Z\setminus\{0, 1\}$. For any $2 \le y \le x$, we have
  \[ \sum_{y < p \le x} \sum_{m \le x/p} \eq(a (m p)^\nu) \lesssim xy^{-1/2} + x^{3/4} q^{1/8} + x^{1/2}(x/q + q)^{1/2}. \]
\end{lemma}

\begin{proof}
  Let~$S_\nu$ be the sum on the left-hand side, and write
  \[ S_\nu \ll \sup_{M\le x/y} S_\nu(M) , 
  \]
  where~$S_\nu(M)$ is the contribution of those~$m\in (M/2, M)$. The contribution of~$M>x^{1/2}q^{-1/4}$ is dealt with using Lemma~\ref{lem:sum-bilin}, which gives
  \[  S_\nu(M) \lesssim xM^{-1/2} + (xM)^{1/2} + x q^{-1/2} + (xq)^{1/2}.  \]
  Taking the supremum over~$M$ satisfying~$x^{1/2} q^{-1/4}<M\le x/y$ gives an acceptable upper bound.

  To deal with the contribution of those~$M\le x^{1/2} q^{-1/4}$ we follow the arguments in~\cite{VauId}, with the choice of parameters~$v = q$ and~$u = x^{1/2}q^{-1/4} M^{-1}$, in a way analogous to the proof of Lemma~\ref{lem:sum-p}.
\end{proof}

Finally we also require the following estimate for double-sums with convolution with two primes for monomial phases.

%% \begin{lemma}% [Convolution with two primes for monomial phases]
%%   \label{lem:sum-p-q}
%%   Let~$\nu\in\Z\setminus\{0, 1\}$. For any $2 \le y \le x$, we have
%%   \[ \sum_{y < p_1 < p_2 \le x} \sum_{m \le x/p_1p_2} \eq(a (m p_1 p_2)^\nu) \lesssim xy^{-1/2} + x^{1/2}(x/q + q)^{1/2}. \]
%%   For~$\nu=1$ and~$y<x^{2/5}$, this estimate holds with~$xy^{-1/2}$ in the right-hand side replaced by~$x^{4/5}$.
%% \end{lemma}

\begin{lemma}% [Convolution with two primes for monomial phases]
  \label{lem:sum-p-q}
  Let~$\nu\neq 0$ and~$j\ge 2$ be integers. For any $2 \le y \le x$, we have
  \begin{align*}
    \sum_{y < p_1 < \dotsb < p_j} & \sum_{m \le x/p_1 \dotsb p_j} \eq(a (m p_1 \dotsb p_j)^\nu) \\
    &  \lesssim  x^{1/2}(x/q + q)^{1/2} + 
    \begin{cases} xy^{-1/2} & \text{if}\ \nu \ne  1,\\
      \min\{xy^{-1/2}, x^{4/5}\} & \text{if}\ \nu =1.
    \end{cases} 
  \end{align*}
\end{lemma}

\begin{proof}
  First assume that either~$\nu \ne 1$ or~$y\ge x^{2/5}$.
  Note that terms with~$p_k=p_\ell$ for some~$k\neq \ell$ can be included in  the sum at a cost~$O(x / \sqrt{y})$.
  We may therefore, by symmetry, relax the conditions~$p_k < p_{k+1}$.
  Next,  we group the variables~$(p_1, m)$ and~$(p_2, \dotsc, p_j)$ together and let
  \[\beta_\ell = \sum_{\substack{p_1\mid \ell \\ p_1>y}} 1, \qquad \gamma_n = \sum_{\substack{p_2, \dotsc, p_k \mid n \\ p_i > y}} 1, \]
  so that our sum can be replaced with 
  \[ S = \sum_{y < n \le x } \sum_{y <\ell \le x / n} \beta_\ell \gamma_n \eq(a (\ell n)^\nu). \]
  Note that~$\abs{\beta_\ell} \le \omega(\ell) \lesssim 1$, and similarly~$\abs{\gamma_n} \le \omega(n)^{k-1} \lesssim 1$, where $\omega(\ell)$ is the number of distinct prime divisors of $\ell$. 
  By Lemma~\ref{lem:sum-bilin}, we deduce
  \begin{equation} \label{eq:Bound S}
    S \lesssim x (y^{-1/2} + q^{-1/2} + (x/q)^{-1/2}),
  \end{equation}
  which gives the claimed estimate for general~$\nu$.

  Suppose next that~$\nu = 1$ and~$y<x^{2/5}$. We use~$j$ times the Heath-Brown identity for primes, as we have done earlier in~\eqref{eq:identity-HB}, which brings us to bound a finite number of sums of the shape
  \begin{equation}\label{eq:HB-multiple}
    \begin{split}
      T =  \underset{(m_1, \dotsc, m_J, n_1, \dotsc, n_{J^*}) \in \cU}{\sum_{\substack{m_1, \dotsc, m_J \\ m_i \le x^{1/5}}}  \sum_{n_1, \dotsc, n_{J^*}}} &\mu(m_1) \dotsb \mu(m_J) \\
      \times V_1(n_1) &\dotsb V_{J^*}(n_{J^*})  \e_q(a m_1 \dotsb m_J n_1 \dotsb n_{J^*}), 
    \end{split}
  \end{equation}
  where~$\cU\subseteq\R^{J+J^*}$ accounts for the various inequalities that involve~$p_i$, and~$V_i(t)$ is either~$1$ or $\log t$. Concerning~$\cU$, we keep only the information that
  \begin{equation} \label{eq:sumpj-log-U}
    \{ (\log m_1, \dotsc, \log {n_{J^*}}):~(m_1, \dotsc, n_{J^*})\in \cU\} 
  \end{equation}
  is a convex set, and in fact an intersection of half-spaces.

  We partition the sum~\eqref{eq:HB-multiple} in  dyadic intervals~$m_i \in [M_i, 2M_i]$ and~$n_i \in [N_i, 2N_i]$, with~$M_i \le x^{1/5}$.
  If there is an index~$i$ such that~$N_i \ge x^{1/5}$, we sum over~$n_i$ first (which we rename into $n$) to get a sum of the shape
  \[  T  \lesssim 
    \sum_{\substack{\ell_1, \dotsc, \ell_{L} \\ \ell_1 \dotsb \ell_L \ll x^{4/5}}} \left| \sum_{n \in I_{\ell_1, \dotsc, \ell_L}} V_i(n) \e_q(a \ell_1 \dotsb \ell_L n) \right|, 
  \]
  where~$L=J+J^*-1=2J$, the set~$I_{\ell_1, \dotsc, \ell_L}$ is an interval by convexity of~\eqref{eq:sumpj-log-U}, and it is contained in~$[1, x/(\ell_1\dotsb \ell_L)]$.
  Therefore, using a Type~I estimate~\cite[Equation~(13.46)]{IwKow} as we have done earlier in~\eqref{eq:typeI-withconvolution}, gives a bound
  \begin{equation} \label{eq:Bound T 1}
    T  \lesssim x^{4/5} + xq^{-1} + q.  
  \end{equation}
  If~$N_i < x^{1/5}$ for all~$i$, then since~$M_i \le x^{1/5}$ as well, we may group variables in such a way as to obtain a Type~II of the shape~\cite[Equation~(13.48)]{IwKow} with~$x^{2/5} \ll M \ll x^{3/5}$, and then we get a bound
  \begin{equation} \label{eq:Bound T 2}
    T  \lesssim  x^{4/5} + xq^{-1/2} + (xq)^{1/2}.  
  \end{equation}
  
  We observe that the desired bound is trivial for $q \ge x$. Otherwise $q \le (xq)^{1/2}$, and 
  taking the weakest of the  bounds~\eqref{eq:Bound T 1} and~\eqref{eq:Bound T 2}, also also recalling that~\eqref{eq:Bound S} also holds for $\nu=1$, we conclude the proof.
\end{proof}

\subsection{Combinatorial decomposition of integers without large prime factors}

The previous results are  used in conjunction with the following two combinatorial decomposition for the indicator function of smooth numbers, which are relevant for small~$y$ and large~$y$ respectively.

\begin{lemma}\label{lem:decomp-ysmall}
  For~$2\le y \le x$ and any bounded map~$f:\N\to \C$, for any positive $w \le x$, for some sequences~$(\alpha_m)$, $(\beta_n)$ of bounded $L_\infty$-norm,  we have
  \[  \sum_{n\in \cS(x, y)} f(n) \lesssim
    w + \sup_{\substack{w\le M \le wy \\ MN \le x}} \abs{\sum_{m\in [M, 2M]} \sum_{n\in [N, 2N]} \alpha_m \beta_n f(mn)} ,  \] 
  where the supremum is over all~$M, N\ge 1$ with the indicated conditions.
  %% , and over all    sequences~$(\alpha_m)$, $(\beta_n)$ of bounded $L_\infty$-norm. 
\end{lemma}

\begin{proof} The argument  is based on the classical combinatorial partition of $y$-smooth integers, see, for example~\cite[p.~1369]{Hmy} or~\cite[Lemma~10.1]{Vau}.
  Then we proceed as in~\cite[Lemma~3.4]{DraTop}. First we bound trivially the contribution of~$n\le w$. Then we factor each~$n\in \cS(x, y)$ uniquely as~$n=km$ with~$w \le k < w P(k)$ and~$P(k) \le p(m)$, where $p(m)$ is the smallest prime divisor of $m$. Finally we separate multiplicatively~$k$ and~$m$ analytically using~\cite[Lemma~13.11]{IwKow}. 
  More precisely,  exactly as in~\cite[Section~9]{FoTe}, we write 
  \[
    \sum_{\substack{n\in \cS(x, y)\\n \ge w}} f(n)  =  \sum_{\substack{w < k \le w P(k)\\P(k) \le y}} \,  
    \sum_{\substack{m \in \cS(x/k,y)\\ p(m) \geq P(k)}}  f(km).  
  \]
  The condition~$p(m) \geq P(k)$ involves integers on both sides of this inequality  of size at most~$y$. We detect this condition by means of~\cite[Lemma~13.11]{IwKow}, getting
  \[ \sum_{\substack{n\in \cS(x, y)\\n \ge w}} f(n)   \lesssim \sup_{t\in \R} \abs{\sfS(t)}, \]
  where
  \[ \sfS (t) = \sum_{\substack{w < k \le w P(k)\\P(k) \le y}} P(k)^{it} \sum_{\substack{m \in \cS(x/k,y)}} p(m)^{-it}  f(km). \]
  We split the sum  $\sfS (t)$ into sums $\sfS(t, K, M)$ over dyadic intervals~$K\le k < 2K$, $M \le m < 2M$ with~$w \le K \le wy$ and~$KM \le x$, and we write accordingly
  \[ \sum_{\substack{n\in \cS(x, y)\\n \ge w}} f(n)  \lesssim  \sup_{\substack{w\le K \le wy \\ KM \le x}} \abs{\sfS(t, K, M)}. \]
  Renaming the variables, we derive the desired result. 
\end{proof}

We also need yet another simple combinatorial identity.

\begin{lemma}\label{lem:decomp-ylarge}
  Let~$f:\N\to\C$ be a bounded map and let~$r$ be a positive integer.
  For~$x^{1/(r+1)}<y\le x^{1/r}$ we have
  \[ \sum_{n\in \cS(x, y)} f(n) = \sum_{n\le x} f(n) + \sum_{j=1}^r (-1)^j \sum_{y < p_1 \le \dotsb \le p_j} \sum_{m\le x/p_1\dotsb p_j} f(mp_1 \dotsb p_j). 
  \]
\end{lemma}

\begin{proof}
  This is the classical Buchstab identity~\cite{Buch}, see, for example,~\cite[Theorem~III.4]{Ten}.
\end{proof}

\section{Proof of Theorem~\ref{thm: Sum S}}

As in the proof of~\cite[Theorem~13]{FoTe}, we choose some parameter $w$ 
to be chosen later subject to~$1 \le w \le x$ (we however do not set $w = x^{1/2}$ as in~\cite{FoTe}).
We apply Lemma~\ref{lem:decomp-ysmall} and use Lemma~\ref{lem:sum-bilin} to bound the resulting sums. 
The four terms in the bound of Lemma~\ref{lem:sum-bilin} can be estimated as 
\begin{align*}
  & M^{1/2}N \le  M^{1/2} (x/M)\le xw^{-1/2}, \\
  & MN^{1/2} \le  M^{1/2} x^{1/2} \le (wxy)^{1/2}, \\
  & MNq^{-1/2} \le x q^{-1/2}, \\
  & MN  (MN/q)^{-1/2} =  (MNq)^{1/2} \le (xq)^{1/2}.
\end{align*}
Hence we obtain 
\[
  S_{a, q}(x, y) \lesssim  w +  x\(w^{-1/2} +  (x/(wy))^{-1/2} + q^{-1/2} +(x/q)^{-1/2}\) .
\]
We now pick~$w = (x/y)^{1/2}$ which indeed satisfies~$1\le w \le x$, and get the bound
\begin{equation}
  S_{a,q}(x, y) \lesssim x\((x/y)^{-1/4} + q^{-1/2} + (x/q)^{-1/2}\).\label{eq:bound-S1-ysmall}
\end{equation}
This proves Theorem~\ref{thm: Sum S} when~$y \le x^{1/5}$, as then~$(x/y)^{-1/4} \le x^{-1/5}$.

We now focus on the range~$x^{1/5} < y \le x$. By Lemma~\ref{lem:decomp-ylarge}  we have
\[
  S_{a,q}(x, y) = \sum_{n\le x} \eq(an) + \sum_{j=1}^5 (-1)^j \sum_{y<p_1\le \dotsb \le p_j} \sum_{m\le x/p_1 \dotsb p_j} \eq(amp_1 \dotsb p_j).
\]
The first sum is trivially~$O(q)$, which is admissible since $q \le (xq)^{1/2}$ for $x \ge q$,  
which we can always assume. The last five sums are bounded, using 
Lemmas~\ref{lem:sum-p} (for $j=1$) and~\ref{lem:sum-p-q}  (for $2 \le j\le 5$), by
\begin{align*}
  \sum_{y<p_1<\dotsb <p_j} \sum_{m\le x/p_1 \dotsb p_j} &\eq(amp_1 \dotsb p_j) \\
  & \lesssim x \(x^{-1/5} + q^{-1/2} + (x/q)^{-1/2}\). 
\end{align*}
This proves Theorem~\ref{thm: Sum S} when~$x^{1/5}<y\le x$.

\section{Proof of Theorem~\ref{thm: Sum Snu}}
\subsection{Preliminary splitting} 
We now assume that~$q$ is prime. Removing the contribution of those integers divisible by~$q$, we get
\[ 
  S_{\nu, a, q}(x, y) = \sum_{\substack{n\in \cS(x, y) \\ \gcd(n, q)=1}} \eq(a n^\nu) + O(x/q).
\]

\subsection{Proof of~\eqref{Snu-bilin}}

The proof of~\eqref{Snu-bilin} is identical to the proof of the bound~\eqref{eq:bound-S1-ysmall}, 
since Lemma~\ref{lem:sum-bilin} holds for any~$\nu$ (and actually, for any non-exceptional trace function).

\subsection{Proof of~\eqref{Snu-bilin+typeI}}

We proceed as in the proof of Theorem~\ref{thm: Sum S}, except that we use Lemma~\ref{lemma:sump-monom} (instead of Lemma~\ref{lem:sum-p}). The details are identical.

\subsection{Proof of~\eqref{Snu-FM}}

Let~$1\le w \le x$ be a parameter. Using Lemma~\ref{lem:decomp-ysmall}, followed by Lemma~\ref{lem:FKM-bilin}, we get
\begin{align*}
  S_{\nu, a, q}(x, y) & \lesssim  xq^{-1/4} + w \\
  & \qquad \qquad \quad + x \sup_{w \le M \le wy} \min\bigl\{M^{-1/2} + x^{-1/2}q^{1/4}M^{1/2}, \\
  & \qquad \qquad \qquad \qquad \qquad  \qquad  \qquad    x^{-1/2}M^{1/2} + q^{1/4}M^{-1/2}\bigr\},
\end{align*}
where we used the symmetry of the bounds of Lemma~\ref{lem:FKM-bilin} with respect to~$M\leftrightarrow N$.

Write~$w = x^\omega$, $y = x^\alpha$, $q=x^\beta$,~$M=x^\mu$, and let
\[ 
  \eta(\mu) = \begin{cases} 
    \min\{\mu/2, 1/2-\beta/4-\mu/2\}, & (0\le \mu\le 1/2), \\ 
    \min\{\mu/2-\beta/4, 1/2-\mu/2\},
    & (1/2<\mu\le 1). \end{cases} \]

%% 1/4--.75/4   = 0.25/4 = 0.625 

Recalling $w \le M \le wy$ we see that only the range $\omega \le \mu \le \omega+ \alpha$ is relevant to us. That is, we are interested in choosing $\omega$ which maximises 
\[
  \kappa = \min_{\omega \le \mu \le \omega + \alpha} \eta(\mu)  .
\]
Note that we dropped the condition $\mu \le 1$ as for $\mu \ge 1$ we have $\eta(\mu)< 0$ and the result is trivial.

The bound above reads
\begin{equation}
  \label{eq:Snu-before-optim}
  S_{\nu, a, q}(x, y) \lesssim x^{1-\beta/4} + x^{\omega} + x^{1-\kappa} 
  %% \sup_{\omega < \mu < \min\{1, \omega + \alpha\}} x^{1-\eta(\mu)}.
\end{equation}

For~$q\le x$, which means~$\beta \le 1$, we maximise $\kappa$ by setting
\[  
  \omega = \begin{cases} 1/2 - \beta/4 - \alpha/2, & (\alpha < \beta/2), \\ 
    (1-\beta)/2, & (\beta/2\le \alpha < \beta), \\ 
    (1-\alpha)/2, & (\beta \le \alpha \le 1), \end{cases}  
\]
which gives in all cases $\kappa = \omega/2$.

Indeed, our optimisation problem has a natural interpretation of fitting a horizontal interval $\cI$
of length $\alpha$ at the maximal height under the plot of the function $\eta(\mu)$ which looks like a union of two symmetric 
peaks, see Figure~\ref{fig:two_peaks}. There are there different regimes which correspond to the 
above choice of $\omega$:
\begin{itemize}
  \item $\cI$ fits entirely inside of one peak, see the solid line on  Figure~\ref{fig:two_peaks};
  \item  $\cI$ fits just under the intersection point of the peaks,  see the dashed line on  Figure~\ref{fig:two_peaks};
  \item  $\cI$ can only be fit strictly below  the intersection point of the peaks and stretches from one edge of the plot to another,  see the dotted line on  Figure~\ref{fig:two_peaks};
\end{itemize}
In fact, it is easy to see that   in the first case there is yet another optimal choice of $\omega$, 
which in the second case case we have infinitely many possibilities. Howeve, since the bound~\eqref{eq:Snu-before-optim} contains the term $x^\omega$ we always select the smallest 
admissible value.

\begin{figure}[h]
  \begin{tikzpicture}
    \begin{axis}[
      axis lines=middle,
      xlabel={$\mu~{ }$},
      ylabel={$\eta(\mu)$},
      xlabel style = {right},
      ylabel style = {left},
      xtick={0, 0.5, 1},
      ytick={\empty},
      domain=0:1,
      samples=200,
      ymin=0, ymax=0.2,
      xmin=0, xmax=1,
      legend pos=north east,
      width=9cm,
      height=5cm,
      ]
      
      \def\beta{0.75}
      
      % Plot for 0 <= mu <= 0.5
      \addplot [
      domain=0:0.5,
      blue,
      thick
      ] {min(x/2, 0.5 - \beta/4 - x/2)};
      
      % Plot for 0.5 < mu <= 1
      \addplot [
      domain=0.5:1,
      blue, % red,
      thick
      ] {min(x/2 - \beta/4, 0.5 - x/2)};
      
      \addplot [
      domain=0.2:0.425,
      red,
      thick
      ] {0.1};
      
      \addplot [
      domain=0.125:0.7,
      red,
      dashdotted, thick
      ] {0.0625};
      
      \addplot [
      domain=0.1:0.9,
      red,
      densely dotted, thick 
      ] {0.05};
      
      \legend{$\eta(\mu)$}
    \end{axis}
  \end{tikzpicture}
  \caption{Three different regimes of $\alpha$ and $\beta$.}
  \label{fig:two_peaks}
\end{figure} 

Noting that~$\omega\le 1/2 \le 1-\beta/4$ in all cases, we get for~$q\le x$ the bound
\[  S_{\nu, a, q}(x, y) \lesssim x q^{-1/4} + x \times \begin{cases} x^{-1/4} q^{1/8} y^{1/4}, & (1\le y \le q^{1/2}), \\ (x/q)^{-1/4}, & (q^{1/2} < y \le q), \\ (x/y)^{-1/4}, & (q<y\le x). \end{cases}  \]
It is easily checked that this coincides with~\eqref{Snu-FM}.

For~$x<q\le x^2$, we assume that~$y < x q^{-1/2}$, for otherwise the claimed bound~\eqref{Snu-FM} is trivial. This translates to~$\alpha < 1-\beta/2$. We optimise the bound~\eqref{eq:Snu-before-optim} by setting
\[  \omega = 1/2 - \beta/4 - \alpha/2,  \]
and we obtain
\[  S_{\nu, a, q}(x, y) \lesssim x q^{-1/4} + x (x/y)^{-1/4} q^{1/8}  \]
in accordance with~\eqref{Snu-FM}.

\subsection{Proof of~\eqref{Snu-FKM}}

First we note that for~$y\le x^{1/3}$, the bound~\eqref{Snu-FM} which we have just proven implies
\[  S_{\nu,a,q}(x, y) \lesssim x (q^{-1/2} + x^{-4/3} q)^{1/8},  \]
which implies~\eqref{Snu-FKM}, reducing the value of~$\delta$ if necessary. We may thus assume~$x^{1/3}<y\le x$.

Assume first that~$x^{1/2}<y\le x$. We use Lemma~\ref{lem:decomp-ylarge} and get
\[  S_{\nu,a,q}(x, y) = O(x/q) + \sum_{\substack{n \le x \\ \gcd(n, q)=1}} \eq(a n^\nu) - \sum_{y<p\le x} \sum_{m\le x/p} \eq(a (mp)^\nu).  \]
Using the Weil bound~\cite{Weil}, coupled with the completing technique~~\cite[Section~12.2]{IwKow}, and periodicity, 
the first sum on the right-hand side is bounded by
\[  \sum_{\substack{n \le x \\ \gcd(n, q)=1}} \eq(a n^\nu)  \lesssim q^{1/2} + x q^{-1/2}.  \]
To bound the second sum, we appeal to the bound~\eqref{eq:trace-mp} of Lemma~\ref{lem:trace-mp}. It follows that for each~$\eps>0$, there exists~$\delta>0$ for which we have
\[  S_{\nu,a,q}(x, y) \lesssim q^{1/2} + xq^{-1/2} + q^{-\delta/4} + q^{(3/4+\eps)\delta} x^{-\delta}.  \]
Reducing~$\delta$ if necessary, the first two terms are absorbed by the last two terms, and we obtain~\eqref{Snu-FKM} for~$y>x^{1/2}$.

The case~$x^{1/3} < y \le x^{1/2}$ similar: upon using Lemma~\ref{lem:decomp-ylarge}, we are to bound an additional sum with~$j=2$, namely
\[  \sum_{y<p_1 \le p_2 \le x} \sum_{m\le x/p_1p_2} \eq(a (mp_1p_2)^\nu),  \]
for which an admissible bound is provided by~\eqref{eq:trace-mp1p2} of Lemma~\ref{lem:trace-mp} .

\section{Comments}
\label{sec:com}

We have already mentioned  that the bound of  Theorem~\ref{thm: Sum S} is nontrivial 
in essentially optimal range  $x \ge q^{1+ \varepsilon}$ 
with an arbitrary fixed $\varepsilon > 0$. 
However, the range where  Theorem~\ref{thm: Sum Snu} 
gives a power saving is unlikely to be the best possible. In fact, for   $\nu\ge 1$ 
one can expect nontrivial bounds starting already from $x \ge q^{1/\nu + \varepsilon}$ 
with an arbitrary fixed $\varepsilon > 0$.  One of the possibilities to extend this range is via the use
of some other bounds on the bilinear sum which appears on the right hand side of~\eqref{eq:BilinSum}, 
exploiting the structure of the argument, 
instead of the generic bound from~\cite[Chapter~VI, Exercise~14.a]{Vinog}. 
For example,   a double application of the H{\"o}lder inequality leads to the following inequality, which 
in several modifications has appeared in a large number of works  and follows the steps in the proof 
of~\cite[Theorem~3]{Kar}. Namely, for any integer $k, \ell \ge 1$, we have
\[
  \abs{ \sum_{\substack{M\le m \le 2M \\ N \le n \le 2N}}  \alpha_{m} \beta_{n} \e_q(a (m n)^\nu) }^{2k\ell} \le q M^{1-1/\ell} N^{1-1/k}
  \(T_k(M) T_\ell(N)\)^{1/k\ell}, 
\]
where $T_k(M)$ is the number of solutions to the congrunce. 
\[ m_1^\nu + \ldots + m_k^\nu \equiv m_{k+1}^\nu + \ldots m_{2k}^\nu \pmod q, 
  \quad M \le m_1, \ldots, m_{2k} \le 2M, 
\]
and similarly for $T_\ell(N)$. Note that with $k = \ell = 1$ this is exactly the bound we
have used to in the proof of Lemma~\ref{lem:sum-bilin}. To estimate  $T_k(M)$ and $T_\ell(N)$ 
for $k, \ell \ge 2$ and $\nu \ge 2$ one can use, for example,~\cite[Theorems~1.1 and~1.2]{KMS}.  
Furthermore, for $\nu \le -1$,  one can also use various  bounds of  Bourgain and Garaev~\cite{BG}, 
Heath-Brown~\cite{HB2} and Pierce~\cite{Pierce}.

Our approach can be adjusted to obtain similar bounds to several variations of the sums
$S_{a,q}(x,y)$ and $S_{\nu, a,q}(x,y)$. For example,   these includes  sums twisted by  multiplicative functions such as
\[
  S_{a,q}(f;x,y)  =   \sum_{n \in \cS(x,y)} f(n) \eq\(an\), 
\]
and more generally
\[
  S_{\nu,a,q}(f;x,y)  =  \sum_{n \in \cS(x,y)} f(n) \eq\(an^\nu\), \quad \nu = \pm 1, \pm 2, \ldots, 
\]
with a multiplicative function $f(n)$. Sums  $S_{a,q}(f;x,y)$ 
have also been studied~\cite[Proposition~1]{dlB1},  see also~\cite[Section~10.2]{dlBGr2}.
If the function $f$ is completely multiplicative, such as a multiplicative character, 
our argument proceeds without 
any changes besides small typographical adjustments and so the bounds 
of Theorems~\ref{thm: Sum S} and~\ref{thm: Sum Snu}  also apply to 
$S_{a,q}(f;x,y)$ and $S_{\nu,a,q}(f;x,y)$ (with an additional factor $\max_{n\le x} |f(n)|$).

%% 
%% For arbitrary multiplicative functions, where we arrive to an analogue of~\eqref{eq: Sum T3}
%% we need to remove integers $m \in \cS(4x/(uz),y)$ with~$p(m) \geq P(k)$ and such that 
%% $f(km) \ne f(k)f(m)$.  and contribution from such $m$ can be estimated separately 
%% (using that $p \in[z/2,z]$). 

\section*{Acknowledgements} 

The authors would like to thank the Royal Institute of Technology and 
Mittag-Leffler Institute for its hospitality and excellent working environment.

During the preparation of this paper, the first author was supported by the joint FWF-ANR project Arithrand: FWF: I 4945-N and ANR-20-CE91-0006 and the second author was also supported by  Australian Research Council Grants DP200100355 and  DP230100530 and by the Knut and Alice Wallenberg Fellowship.

\end{document}